\documentclass[reqno]{amsart}
\usepackage{graphicx}
\usepackage{latexsym}
\usepackage{amssymb,amsthm,amsmath}
\usepackage[active]{srcltx}
 \makeatletter
\renewcommand*\subjclass[2][2000]{%
  \def\@subjclass{#2}%
  \@ifundefined{subjclassname@#1}{%
    \ClassWarning{\@classname}{Unknown edition (#1) of Mathematics
      Subject Classification; using '2010'.}%
  }{%
    \@xp\let\@xp\subjclassname\csname subjclassname@#1\endcsname
  }%
}
 \makeatother
\usepackage{enumerate,url,amssymb,  mathrsfs}%, pdfsync}% ,showkeys, pdfsync}
\usepackage{todonotes}

\usepackage{color}
\newtheorem{theorem}{Theorem}[section]
\newtheorem{lemma}[theorem]{Lemma}
\newtheorem*{lemma*}{Lemma}

\theoremstyle{definition}
\newtheorem{definition}[theorem]{Definition}

\theoremstyle{remark}
\newtheorem{remark}[theorem]{Remark}

\numberwithin{equation}{section}

\def\mxi{\boldsymbol{\xi}}
\def\cal{\mathcal}

\newcommand{\pa}{\partial}

\newcommand{\R}{\mathbb{R}}

\newcommand{\ddiv}{\mathrm{div}}

\def\XXint#1#2#3{{\setbox0=\hbox{$#1{#2#3}{\int}$}
\vcenter{\hbox{$#2#3$}}\kern-.5\wd0}}

\def\le{\leqslant}
\def\ge{\geqslant}
\def\mx{{\bf x}}

\def\mff{{\mathfrak f}}
\def\eps{\varepsilon}

\newcommand{\N}{\mathbb N}

\usepackage[normalem]{ulem}

\begin{document}
\title[A vanishing dynamic capillarity limit]{A vanishing dynamic capillarity limit equation with discontinuous flux}
\author{M.\ Graf}\address{Melanie Graf, Department of Mathematics
University of Washington
C-138 Padelford Hall
Seattle, WA 98195-4350}\email{mgraf2@uw.edu}
\author{M.\ Kunzinger}\address{Michael Kunzinger, University of Vienna, 
Faculty of Mathematics, Oskar Morgenstern--Platz 1, 1090 Wien, Austria}\email{michael.kunzinger@univie.ac.at}
\author{D.\ Mitrovic$^*$}\thanks{$*$ The corresponding author.}\address{Darko Mitrovic,
Faculty of Mathematics, University of
Vienna, Oskar Morgenstern--Platz 1, 1090 Wien, Austria}\email{darko.mitrovic@univie.ac.at}
\author{Dj.\ Vujadinovic}\address{Djordjije Vujadinovic,
Faculty of Mathematics, University of
Montenegro, Cetinjski put bb, 81000 Podgorica, Montenegro}\email{djordjijevuj@ucg.ac.me}



\begin{abstract} 
We prove existence and uniqueness of a solution to the Cauchy problem corresponding to the {\bf dynamics capillarity} equation
\begin{equation*}
\begin{cases}
\partial_t u_{\eps,\delta} +\mathrm{div} {\mathfrak f}_{\eps,\delta}(\mx, u_{\eps,\delta})=\eps \Delta u_{\eps,\delta}+\delta(\eps) \pa_t \Delta u_{\eps,\delta}, 
\ \ {\bf x} \in M, \ \ t\geq 0\\
u|_{t=0}=u_0(\mx).
\end{cases}
\end{equation*}
Here, $\mff_{\eps,\delta}$ and $u_0$ are smooth functions while $\eps$ and $\delta=\delta(\eps)$ are fixed constants. Assuming $\mff_{\eps,\delta} \to \mff \in L^p( \R^d\times \R;\R^d)$ for some $1<p<\infty$, strongly as $\eps\to 0$, we prove that, under an appropriate relationship between $\eps$ and $\delta(\eps)$ depending on the regularity of the flux $\mff$,  the sequence of solutions $(u_{\eps,\delta})$ strongly converges in $L^1_{loc}(\R^+\times \R^d)$ towards a solution to the conservation law 
$$
\partial_t u +\mathrm{div} {\mathfrak f}(\mx, u)=0.
$$ 
The main tools employed in the proof are the Leray-Schauder fixed point theorem for the first part and reduction to the kinetic formulation combined with recent results in the velocity averaging theory for the second. These results have the potential to generate a stable semigroup of solutions to the underlying scalar conservation laws different from the Kruzhkov entropy solutions concept. 
\end{abstract}
\maketitle

\section{Introduction and Notation}\label{intro}

Flow in a two-phase porous medium is governed by the Darcy law \cite{bear}

\begin{equation}
\label{D}
q=-K(S) \left( \nabla p +\rho g e_d \right),  
\end{equation} where $e_d=(0,\dots,0,1),$ is the direction of gravity. The quantity $S$ is the saturation, $p$ is the pressure, and (the vector) $q$ is the flow velocity of the wetting phase (usually water, while the non-wetting one is oil or a gas). 

The Darcy law represents conservation of momentum and, in order to close the system, we also need the conservation of mass
\begin{equation}
\label{mc}
\pa_t S+ \ddiv q =0.
\end{equation} In the two-dimensional situation, we have three equations, given by \eqref{D} and \eqref{mc}, while we have four unknowns (two velocity components, saturation and pressure). Therefore, usually one assumes a constitutive relation between the pressure $p$, the capillary pressure $P_c$ (equal to differences of pressures between wetting and non-wetting phases), and the saturation $S$. If it is assumed that the capillary pressure is (almost) constant, one can derive the Buckley-Leverett equation (a scalar conservation law derived in \cite{BL}):
\begin{equation}
\label{BLe}
\pa_t S+\pa_x f(S)=0,
\end{equation}
for $f(S)=\frac{S^2}{S^2+A(1-S)^2}$ and an appropriate constant $A$.

If we assume that $P_c$ is "static" (independent of the $t$-derivative of $S$) in the sense that $p_c=p_c(S)$ then we arrive at a parabolic perturbation of the Buckley-Leverett equation \cite{bear}. However, both of the models (the standard Buckley-Leverett or the one perturbed by a parabolic term) appear to give results inconsistent with certain fairly simple experiments \cite{DC}. Namely, if we take a thin tube filled with dry sand and dip it in water at a constant rate at one side of the tube and then measure the concentration of the water as a function of time, for certain dipping rates the concentration will not be monotonic. This phenomenon,
called overshoot, had been noticed long ago and is a sub-effect of the fingering effect \cite{H}. Namely, the tips of the fingers that appear when water is penetrating into dry land have larger water concentrations than the body of the finger and, interestingly, the tips of the fingers almost do not change their shape or size. On the other hand, an entropy solution to the Buckley-Leverett equation (in the sense of Kruzhkov \cite{Kru} or more precisely Oleinik \cite{O} since we are dealing with a Riemann problem) or the parabolic perturbation of the Buckley-Leverett equation with the Riemann initial data
\begin{equation}
\label{riemann}
S(0,x)=\begin{cases}
S_L, & x<0\\
0, & x>0
\end{cases}
\end{equation} is monotonic (see e.g.\ the introduction of \cite{vDPS}). Therefore, many attempts to explain the gap between the standard theory (provided by the Buckley-Leverett or Richards \cite{R} equations) have been put forth recently. A purely mathematical approach can be found in \cite{KMN}, where 
$\delta$-type-solutions to \eqref{BLe} are constructed and $\delta$-distributions 
appearing as a part of the solution are explained as an inadequacy of the model. 
In \cite{PRE}, a model is suggested that involves the notion of non-local energy, which eventually leads to a fourth-order equation whose solution (with the Riemann initial data given above) has a shape corresponding to the experimental results from \cite{DC}. 

In \cite{vDPS} one can find an interesting approach explaining the discrepancy between theory and experiment described above, based on
the dynamics capillarity concept introduced in \cite{GH1, GH2}, which has drawn a lot of interest (especially after the publication of \cite{vDPS}; see also \cite{AndrCances}). 
Namely, in \cite{GH1, GH2} it was supposed that the capillary pressure depends not only on the saturation $S$, but also on the time derivative of the 
saturation (of the wetting phase):
$$
P_c=p_c(S)-\phi \tau \frac{\pa S}{\pa t}.
$$ Taking this into account and proceeding along the lines of deriving the Buckley-Leverett equation, one reaches a nonlinear pseudo-parabolic equation  (\cite[(1.16)]{vDPS}) which, after linearization of higher order terms reduces to
\begin{equation}
\label{BL3}
\pa_t S+\pa_x f(S)=\eps \pa_{xx}S+\tau \delta \pa_{xx} \pa_t S,
\end{equation}
where $\eps$ and $\delta$ ($\delta=\eps^2$ in \cite{vDPS}) are small parameters, while $\tau$ is a fixed constant. By analyzing possible traveling wave solutions $S(\frac{x-ct}{\eps})$, $S(-\infty)=S_L$, $S(+\infty)=S_R$, where the constant $c$ is given by the Rankine-Hugoniot condition $c=\frac{f(S_L)-f(S_R)}{S_L-S_R}$, the authors arrive at solutions constructed from elementary waves (i.e., solutions consisting of shock waves and rarefaction waves) to \eqref{BLe}, \eqref{riemann} that are non-standard (i.e., non-admissible in the sense of Oleinik). Moreover, for some values $S_L$, such a solution exhibits an overshoot-type phenomenon (see \cite[Figure 7]{vDPS}). 

In this paper, we shall make a step forward in the sense that we shall consider a multi-dimensional generalization to \eqref{BL3} with a flux $\mff=\mff(\mx,\lambda)$ explicitly depending on the space variable and we shall analyze the dynamics capillarity limit to \eqref{BL3} as $\eps\to 0$ for arbitrary initial data. 

The $\mx$-dependence of the flux means that we assume that the medium in which we consider the phenomenon is heterogeneous, i.e.\ that it has different properties at different points (e.g.\ since in some parts of the medium we have sand while in some others clay). Moreover, we shall assume that the flux is discontinuous with respect to the space variable $\mx$, which means that the medium experiences abrupt (discontinuous) changes in its properties (one can imagine that we have sand which is highly permeable adjacent to the clay layer, which is weakly permeable --- permeability is discontinuous in such a medium). Let us remark here that, stipulated by different applications, evolutionary equations with discontinuous flux have attracted considerable attention recently. For a (non-exhaustive) selection of recent
results, cf.\ \cite{AKR, AM, crasta, kar3, Pan} and references therein.

%Formally written, 
Explicitly, we consider the conservation law

\begin{equation}
\label{cl-1}
\pa_t u+\ddiv \mff(\mx,u)=0,
\end{equation} where

\begin{itemize}

\item[(C1)] $\mff = (f_i)_{i=1}^d = (\mx,\lambda)\mapsto \mff(\mx,\lambda)$, 
 $f_i \in L^1(\R^d\times \R) \cap L^p(\R^d\times \R)$ and $\pa_{\lambda} f_i \in L^p(\R^d\times \R)$ for some $1<p<\infty$ 
 and all $i=1,\dots, d$;

\item[(C2)] $\ddiv_{\mx} \mff(\mx,\xi) \in {\cal M}(\R^d\times \R)$, where ${\cal M}(\R^d\times \R)$ is the space of Radon measures, and there exists a constant $\beta>0$ and a finite Radon measure $\mu \in {\cal M}(\R^d)$ such that (in the sense of measures) 
\begin{equation*} 
|\ddiv_{\mx} \mff(\mx,\xi)|\leq \frac{\mu(\mx)}{1+|\lambda|^{1+\beta}};
\end{equation*}

\item[(C3)] $\|\sup\limits_{\lambda \in \R} |\pa_{\lambda}f_i(\cdot,\lambda)| \|_{L^1(\R^d)} <\infty$ for all $i=1,\dots, d$;

\end{itemize} 
We then perturb \eqref{cl-1} by (a linearized) vanishing dynamic capillarity limit: 

\begin{equation}
\label{vc-1}
\pa_t u_{\eps}+\ddiv \mff_\eps (\mx,u_\eps)=\eps  \Delta u_\eps +\delta (\eps)\Delta \pa_t u_\eps.
\end{equation}  
Here, $\mff_\eps := K_\eps \cdot  \mff \star \omega_{n(\eps )}$, where componentwise convolution $\mff \star \omega_{n(\eps )}$ is a regularization of $\mff$ and $K_\eps=K_\eps(\mx,\lambda)$ is a bounded family of compactly supported functions equal to one on the ball $B(0,1/\eps)\subset \R^{d+1}$ and such that $\|\nabla K_\eps\|_{L^\infty(\R^{d+1})} \leq 1$. {We remark here that \eqref{vc-1} reminds on the diffusion-dispersion regularization (see e.g. \cite{AMP, Lfl-b}), but it contains $t$-derivative in the third order term which makes significant difference between the two situations. We will comment on this in more details later.} 
 
More precisely, we suppose that $\omega\in C^\infty_c(\R^{d+1})$ is of the form 
$\omega(\mx,\lambda)=\omega^{(1)}(\mx)$ $\omega^{(2)}(\lambda)$ with $\omega^{(1)}, \omega^{(2)}$ test functions with 
unit integral on $\R^d$ and $\R$, respectively, and
\begin{equation}\label{mollifying}
\omega_{n(\eps)}(\mx,\lambda) \equiv \omega^{(1)}_{n(\eps)}(\mx)\omega^{(2)}_{n(\eps)}(\lambda) =
n(\eps)^{-d} \omega^{(1)}(\mx/n(\eps))n(\eps)^{-1} \omega^{(2)}(\lambda/n(\eps))
\end{equation} with an appropriate $n(\eps)$ such that  $n(\eps )\to 0$ 
as $\eps\to 0$ (the form of $n(\eps)$ will be precisely determined later). 
Note that
\[
 K_\eps \mff_\eps  \to \mff  \text{ strongly in } L^p(\R^d\times \R) \text{ as } \eps\to 0.
\]
We supplement \eqref{cl-1} with the initial data
\begin{equation}
\label{id}
u|_{t=0}=u_0(\mx)\in L^1(\R^d) \cap L^\infty(\R^d),
\end{equation} while we take 
\begin{equation}
\label{id-pert}
u_\eps\big|_{t=0}=u_{0}^\eps(\mx) \to u_0(\mx) \ \ \text{strongly in} \ \ L^p(\R^d),
\end{equation} where, for every $i=1,\dots,d$,
\begin{equation}
\label{cond}
\|u_{0}^{\eps}\|_{L^{2}(\R^{d})}+n(\eps) \|\nabla u_{0}^{\eps}\|_{L^{2}(\R^{d})}+n(\eps)^{2}\|\nabla u^\eps_{0x_{i}}\|_{L^{2}(\R^{d})}\leq C_{0}.
\end{equation} 
%Under the assumptions $\delta(\eps)={\cal O}(\eps^2)$, 
It is not difficult to see that $u_{0}^\eps=u_0\star \omega^{(1)}_{n(\eps)}$, where $\omega^{(1)}_{n(\eps)}$ is as above, satisfies \eqref{cond}. We shall, however, use \eqref{cond} as an assumption.

With regards to pseudo-parabolic equations, we have already explained their importance in the porous media theory. Besides a possible justification of the experimental results from \cite{DC}, there is a confirmed application to the seepage
of homogeneous fluids through a fissured rock \cite{2Y}. Also, they describe the unidirectional propagation of nonlinear, dispersive long waves \cite{3Y, 29Y} 
(where $u$ is typically the amplitude or velocity), or population dynamics \cite{22Y} (where $u$ represents the population density). 

In \cite{26, 30Y}, the authors investigated the initial-boundary value problem and the Cauchy problem for linear pseudo-parabolic equations and established the existence and uniqueness of solutions. As for the non-linear variants, one can find numerous results even for singular pseudo-parabolic equations and degenerate pseudo-parabolic equations (see e.g. \cite{CYW, KNS, Mi, 12Y, 17Y,  Pt} and the references therein). Together with the existence and uniqueness results, among the given references, one can also find properties of solutions, such as asymptotic behavior and regularity. However, we are not aware of any corresponding results for the Cauchy problem for an equation of type \eqref{BL3}. 

This problem will be considered in Section \ref{sec-2}. We shall use an approach that is characteristic for the theory of wave equations \cite{Ach}: we shall apply the Fourier transform with respect to $\mx$, solve the ordinary differential equation so obtained with respect to $t$ and prove the following theorem (the constants $\eps$ and $\delta$ are omitted for simplicity):
\begin{theorem} \label{unique_sol_par}
%\label{ex-un}
There exists a unique solution to 
\begin{equation}
\label{main-eq}
u_{t}+\ddiv(\mff(\mx, u)) =\triangle u+\partial_{t}\triangle u, \ \ (t,x)\in [0,T)\times \R^d,
\end{equation} where $T>0$, $\mff \in C^\infty_c(\R^d\times  \R)$ is a bounded function, supplemented with (merely) the initial condition
\begin{equation}
\label{id-nonpert}
u|_{t=0}=u_0(\mx) \in C_c^\infty(\R^d),
\end{equation} which belongs to $L^2((0,T)\times \R^d)\cap C^\infty((0,T)\times \R^d)$.
\end{theorem}

%%%%%

In the next step (Section 3), we are going to let the perturbation parameter $\eps\to 0$ in \eqref{vc-1}.  This type of problem for conservation laws (when
there is not only vanishing viscosity, but also third or higher order perturbation) was first addressed in \cite{Sch} and received considerable attention after that. A thorough analysis of this kind of limit {\it in the one-dimensional situation with a regular flux} can be found in \cite{Lfl-K}, where the theory of non-classical shocks for conservation laws was essentially initiated. The standard approach here is to rewrite the equation under consideration in the kinetic formulation or using Young measures (which is essentially equivalent), and then applying velocity averaging results \cite{AMP, HTz}, compensated compactness (in one dimensional situations) \cite{coclitte1, coclitte, Sch}, or Di Perna techniques involving Young measures \cite{Dpe}, as done in \cite{Lfl-b} and many others. Again we note that this list of citations is far from complete. The problem in our case is the low regularity of the flux function $\mff$ (it can be discontinuous with respect to $\mx\in \R^d$) as well as the multidimensional character of our problem, which prevents us from using the Young measures and compensated compactness approach here (see item (iii) below and Remark \ref{39} for the case of a regular flux). We also remark that, to the best of our knowledge, the only diffusion-dispersion type result for equations with discontinuous coefficients is \cite{HKM}.

Moreover, unlike the situation that is typical in the case of the diffusion-dispersion limit (\cite{Sch, AMP, Lfl-b} etc.), where the existence of a solution to the perturbed equation is assumed together with all necessary properties of the solution, we have proved in the previous section the existence of the functional sequence whose convergence we analyze. 

As for the velocity averaging theory, most of the results are given in the case of a homogeneous flux \cite{DLM, GSR, PS, TT} or a flux for which $p \geq 2$ \cite{Ger, LM2}. We have recently proved \cite{LM5} the velocity averaging lemma in the case $p>1$ and this will enable us to prove the strong convergence of the sequence $(u_\eps)$ of solutions to \eqref{vc-1}, \eqref{id-pert} along a subsequence towards a weak solution to \eqref{cl-1}. The following statements are the main results of the paper:

\begin{itemize}

\item[(i)] If $\delta=o(\eps^2)$  and  
\begin{equation}
\label{neps}
\frac{\sqrt{\delta}}{\eps n(\eps)^{d+2}}
\to 0
\end{equation} as $\eps\to 0$,
then the family $(u_\eps)$ {contains a strongly converging subsequence} in $L^1_{loc}(\R^+\times \R^d)$ under a suitable non-degeneracy condition (see Definition \ref{def-non-deg});

\item[(ii)] If $\delta={\cal O}(\eps^2)$ and $\frac{\sqrt{\delta}}{\sqrt{\eps} n(\eps)^{d/2+2}}\to 0$ as $\eps\to 0$, and the $\lambda$-derivative of the flux $\pa_\lambda \mff$ is bounded in addition to condition
(C3), then the family $(u_\eps)$ is strongly precompact in $L^1_{loc}(\R^+\times \R^d)$ under the non-degeneracy condition;

\item[(iii)] If $\delta=o(\eps^2)$ and $\frac{\sqrt{\delta}}{\sqrt{\eps} n(\eps)^{d/2+2}}\to 0$ as $\eps\to 0$, and if the flux $\mff \in C^1(\R^d\times \R)$, 
then the family $(u_\eps)$ converges strongly in $L^1_{loc}(\R^+\times \R^d)$ towards the entropy solution to \eqref{cl-1}, \eqref{id} (no non-degeneracy condition is needed).

\end{itemize}

We conclude the introduction by noting that the subsequence of $(u_\eps)$ given in (i) does not converge toward an entropy admissible solution to the underlying conservation law since its flux is not smooth and for such fluxes the well posedness theory is not developed yet in full generality (compare with item (iii) above). 

Also, in the case (ii), even when the flux is smooth, we do not necessarily have convergence toward the Kruzhkov entropy solution (see \cite{vDPS} for stepwise initial data). On the other hand, the approximating procedure actually models a physical situation in which diffusion and capillary limits have equal effect on the process and, as any physical phenomenon in the macro-world is well-posed, it is therefore expected to generate a stable semigroup of solutions to the underlying conservation law. We will deal with this question in a future research.
 
The paper is organized as follows In Section 2, we prove existence and uniqueness for the pseudo-parabolic problem \eqref{main-eq}, \eqref{id-nonpert}. In Section 3, we derive necessary estimates from \eqref{vc-1}, \eqref{id-pert} and use it to prove the strong convergence toward a weak solution to the underlying conservation law.

\section{Existence and uniqueness of the solution to the pseudo-parabolic equation \eqref{main-eq} with \eqref{id-nonpert}} \label{sec-2}
Throughout this section we suppose that $\mff$ is {\bf smooth and compactly supported}. Under this assumption, we want 
to show that \eqref{main-eq} with the initial condition \eqref{id-nonpert} has a unique solution. 
The strategy for solving this problem is to define a mapping ${\cal T}: L^2([0,T)\times \R^d) \to L^2([0,T)\times \R^d)$ such that for every $v\in L^2([0,T)\times \R^d)$ the function 
\begin{equation}
\label{T}
u={\cal T}(v)
\end{equation} represents a solution to 
\begin{equation}
\label{main-eq-1}
u_{t}+\ddiv(\mff(\mx, v)) =\triangle u+\partial_{t}\triangle u, \ \ (t,x)\in [0,T)\times \R^d,
\end{equation} 
with the initial conditions \eqref{id-nonpert}. Generally, for $u_0\in L^2(\R^d)$, we are seeking weak solutions $u$ of 
this initial value problem in the sense that, 
for any smooth test function $\varphi$ with compact support in $[0,T)$, we require 
\[
\int_0^T\int_{\R^d} u\varphi_t + f(x,v)\ddiv \varphi - u\Delta\varphi + u\partial_t  \Delta\varphi\,dxdt 
+ \int_{\R^d} u_0(\varphi(0,x)+\Delta\varphi(0,x))\,dx=0.
\]

Then, we shall prove that the mapping ${\cal T}$ possesses a fixed point,
which will turn out to be the solution to \eqref{main-eq}, \eqref{id-nonpert}. 

To this end, we need the following consequence of the Leray-Schauder fixed point theorem
(cf.\ \cite[Th.\ 11.3]{GT}):
\begin{theorem}
\label{fixed}
 Let ${\cal T}$ be a compact mapping  of a Banach space $\mathcal{B}$ into
 itself and suppose that there exists a constant $C$ such that
\begin{equation}
\label{cond-fp}
\|u\|_{{\cal B}}\leq C
\end{equation} for all $u\in {\cal B}$ and $\sigma \in [0,1]$ satisfying $u=\sigma {\cal T}u$. Then ${\cal T}$ has a fixed point, that is, ${\cal T}u=u$ for some $u\in \mathcal{B}$.
\end{theorem}

Finally, let us fix our conventions for Fourier transform and inverse Fourier transform. 
For $u\in L^1(\R^d)\cap L^2(\R^d)$, the Fourier transform of $u$ is
$$
{\cal F}(u)(\mxi)=\hat{u}(\mxi)=\int_{\R^d} e^{- i \mx \cdot \mxi} u(\mx)\, d\mx
$$ and the inverse Fourier transform is given by
$$
{\cal F}^{-1}(u)(\mx)=\check{u}(x)= \frac{1}{(2\pi)^d}\int_{\R^d} e^{i \mx \cdot \mxi} u(\mxi)\, d\mxi,
$$ 
where $\cdot$ denotes the scalar product and $i$ is the imaginary unit.

We begin by proving existence and uniqueness of solutions to \eqref{main-eq-1}, \eqref{id-nonpert}. 

\begin{lemma}
\label{ex-v}
Let $v\in L^2([0,T]\times \R^d)$.  Then for any $u_0\in L^2(\R^d)$ 
 there exists a unique solution $u\in L^2([0,T]\times \R^d)$ 
%\melanietodo{replaced $H^1$ by $L^2$, show $H^1$ estimates only in the next lemma} 
to  \eqref{main-eq-1} with $u|_{t=0}=u_0$.
\end{lemma}
\begin{proof}
After applying the Fourier transform with respect to $\mx \in \R^d$ to \eqref{main-eq-1}, 
we obtain 
\begin{equation*}
\pa_t \hat{u}(t,\mxi)+i \mxi \cdot \widehat{\mff(\cdot,v(t,\cdot))}(\mxi)=-|\mxi|^2 \hat{u}(t,\mxi) - |\mxi|^2  \pa_t \hat{u}(t,\mxi),
\end{equation*}
where $|\mxi|=(\xi_1^2+\dots \xi_d^2)^{\frac{1}{2}}$. We rewrite the last equation in the form 
%\melanietodo{fixed $|\mxi |^2$ factor here and in all the following calculations (I think, please check!)}

\begin{equation*}
\frac{d \hat{u}(t,\mxi)}{dt} +\frac{|\mxi|^2}{{1+|\mxi|^2}}\hat{u}(t,\mxi) =-\frac{i \mxi \cdot \widehat{\mff(\cdot,v(t,\cdot))}(\mxi)}{1+|\mxi|^2}
\end{equation*} 
It follows from this that any weak solution to \eqref{main-eq-1} has $t$-derivative in $L^2$ as well, hence in particular is
continuous with respect to $t$ and therefore possesses a classical trace on $t=0$. 
Solving the ODE with the initial data $\hat{u}(0,\mxi)=\hat{u}_0(\mxi)$ we arrive at
\begin{align}
\label{sol-v}
\hat{u}(t,\mxi)=e^{-\frac{|\mxi|^2 t}{1+|\mxi|^2}} \left( \hat{u}_0(\mxi)-\int_0^t\frac{i \mxi \cdot \hat{\mff}}{1+|\mxi|^2} e^{\frac{|\mxi|^2 t'}{1+|\mxi|^2}}\,dt'\right),
\end{align}
where $\hat{\mff}=\widehat{\mff(\cdot,v(t,\cdot))}(\mxi)$. By finding the inverse Fourier transform here with respect to $\mxi \in \R^d$, we 
indeed obtain a weak solution, so combined with the above considerations both existence and uniqueness follow. 
\end{proof}

%As a next step, we derive estimates necessary to apply the Lions-Aubin Lemma and Theorem \ref{fixed}.

\begin{theorem}
\label{estimates-1}
For any function $v\in L^2([0,T]\times \R^d)$, the solution to \eqref{main-eq-1}, \eqref{id-nonpert} belongs to $H^1((0,T)\times \R^d)$.
\end{theorem} 
\begin{proof}
Let $u$ be the solution to \eqref{main-eq-1}, \eqref{id-nonpert} defined in Lemma \ref{ex-v}. We need to estimate $\pa_{x_j} u$, $j=1,\dots,d$ and $\pa_t u$, where $\pa_{x_j}$, $j=1,\dots, d$ and $\pa_t$ are the weak derivatives of $u$. We have according to the Plancherel theorem (below, we take the Fourier transform with respect to $\mx\in \R^d$) and the Cauchy-Schwartz inequality
\begin{equation}\label{est-x}
\begin{split}
\|\pa_{x_j} & u\|_{L^2([0,T]\times \R^d)} = \| i \xi_j \hat{u}\|_{L^2([0,T]\times \R^d)} \\
&\leq \| \xi_j \hat{u}_0 \|_{L^2([0,T]\times \R^d)}+\left\|\int_0^t\frac{i \xi_j \mxi\cdot \hat{\mff}}{1+|\mxi|^2}e^{\frac{|\mxi|^2 t'}{1+|\mxi|^2}}dt'\right\|_{L^2([0,T]\times \R^d)} 
\\& \leq \| \xi_j \hat{u}_0 \|_{L^2([0,T]\times \R^d)}+\left( \int_0^T \int_{\R^d} \left| \int_0^t \frac{i \xi_j \mxi\cdot \hat{\mff}}{1+|\mxi|^2}e^{\frac{|\mxi|^2 t'}{1+|\mxi|^2}}dt' \right|^2 d\mxi dt\right)^{1/2}
\\& \leq \| \xi_j \hat{u}_0 \|_{L^2([0,T]\times \R^d)}+\left( \int_0^T \int_{\R^d} t \int_0^t \frac{\xi^2_j |\mxi\cdot \hat{\mff}|^2}{(1+|\mxi|^2)^2}e^{\frac{2|\mxi|^2 t'}{1+|\mxi|^2}}dt' d\mxi dt\right)^{1/2}\\
&\leq \| \xi_j \hat{u}_0 \|_{L^2([0,T]\times \R^d)}+T e^{T} 
\left(\int_0^T\int_{\R^d} \frac{\xi_j^2 \, |\mxi|^2 \, |\hat{\mff}|^2}{(1+|\mxi|^2)^2} d\mxi dt\right)^{1/2}\\
&\leq  \| \xi_j \hat{u}_0 \|_{L^2([0,T]\times \R^d)}+T e^{T} \left(\int_0^T\int_{\R^d} |\mff(\mx,v(t,\mx))|^2 d\mx dt\right)^{1/2},
\end{split}
\end{equation}where in the last step we used the Plancherel theorem again. From here, since $f$ is compactly supported with respect to $\mx \in \R^d$ and bounded, we conclude that $\pa_{x_j} u \in L^2([0,T]\times \R^d)$. 

As for the $t$-derivative, we have
\begin{align}
\label{est-t}
\pa_t \hat{u}=
-\frac{|\mxi |^2}{1+|\mxi|^2} e^{\frac{-|\mxi |^2 t}{1+|\mxi|^2}} \left(\hat{u}_0(\mxi)-\int_0^t\frac{i\mxi \cdot \hat{\mff}}{1+|\mxi|^2} dt' \right)+\frac{i\mxi \cdot \hat{\mff}}{1+|\mxi|^2}, 
\end{align} and therefore, repeating the procedure giving us the estimates on $\pa_{x_j} u$, we get
$$
\|\pa_t u \|_{L^2([0,T]\times \R^d)}=\|\pa_t \hat{u} \|_{L^2([0,T]\times \R^d)} \leq C
$$ for a constant $C\in \R$.
\end{proof} 
Based on this we can proceed to applying Theorem \ref{fixed} to prove existence of a solution to \eqref{main-eq}, \eqref{id-nonpert}. 
%We have the following statement.

\begin{theorem}
\label{ex-un}
There exists a unique solution to \eqref{main-eq}, \eqref{id-nonpert}. This solution belongs to $L^2([0,T]\times \R^d)\cap C^\infty([0,T]\times \R^d)$.
\end{theorem}
\begin{proof} Let us first take a sequence of balls $B_n=B(0,n)$ centered at $0$ and of radius $n$. Set $u_0^n(\mx)=u_0(\mx) \chi_{B_n}(\mx)$, where $\chi_{B_n}: \R^d\to [0,1]$ is a smooth regularization of the characteristic function of the set $B_n$ such that $\chi_{B_n}(\mx)=1$ for $\mx \in B(0,n-1)$ 
and $\chi_{B_n}(\mx)=0$ for $\mx \notin B(0,n)$. 
Then we define the mapping ${\cal T}_n: L^2([0,T]\times B_n) \to  L^2([0,T]\times B_n)$ by
$$
{\cal T}_n (v) :=\chi_{B_n} {\cal T} (v\chi_{B_n}),
$$ where ${\cal T}$ is the operator from \eqref{T} with the initial data $u_0 \chi_{B_n}$ and we extend $v\chi_{B_n}$ by 
zero so it is defined on $[0,T]\times \R^d$. 
This map is continuous since ${\cal T}$ is continuous. In fact, by the proof of Theorem \ref{estimates-1},
$T$ is in fact continuous as a map from $L^2([0,T)\times \R^d)$ to $ H^1((0,T)\times \R^d)$. 
So we can use Theorem \ref{estimates-1} and the Rellich theorem (keeping in mind that $B_n$ is a bounded set) 
to conclude that ${\cal T}_n$ is a compact mapping. 

Let us check condition \eqref{cond-fp} from Theorem \ref{fixed}. If a function $u\in L^2([0,T]\times B_n)$ satisfies $u=\sigma {\cal T}_n u$, then the function $\tilde{u}:=\sigma {\cal T} (u \chi_{B_n})$ satisfies $\chi_{B_n} \tilde{u}(t,\mx)=u(t,\mx)$ for $\mx\in B_{n}$ and is a solution 
of the Cauchy problem 
\begin{equation*}
\begin{split}
\tilde{u}_{t}+\sigma \ddiv(\mff(\mx, u \chi_{B_n})) &=\triangle \tilde{u}+\partial_{t}\triangle \tilde{u}, \ \ (t,x)\in [0,T)\times \R^d\\
\tilde{u}(0,\mx) &= \sigma u_0(\mx) \chi_{B_n}(\mx)
\end{split} 
\end{equation*} 
If we apply Fourier transform with respect to $\mx\in \R^d$ here and solve the ODE so obtained, 
we arrive at an expression analogous to \eqref{sol-v}: 

\begin{align*}
\hat{\tilde{u}}(t,\mxi)=\sigma e^{-\frac{t|\mxi |^2}{1+|\mxi|^2}} \left( \widehat{u_0\chi_{B_n}}(\mxi)-
\int_0^t\frac{i \mxi \cdot \hat{\mff}}{1+|\mxi|^2} e^{\frac{t' |\mxi |^2}{1+|\mxi|^2}}\,dt'\right),
\end{align*}
where $\hat{\mff}={\cal F}(\mff(\mx,u(t,\mx)\chi_{B_n}(t,\mx))$ is the Fourier transform with respect to $\mx$. 
Now
$ \| u \|_{L^2([0,T]\times B_n)} \leq   \| \tilde{u} \|_{L^2([0,T]\times \R^d)}$ 
and repeating the procedure from \eqref{est-x} and \eqref{est-t}, we conclude that \eqref{cond-fp} holds for the mapping ${\cal T}_n$ (keeping in mind that $n$ is fixed). 

Thus, for every $n\in \N$ there exists a function $u_n\in L^2([0,T]\times B_n)$ solving 
\begin{equation*}
\begin{split}
\pa_t{u_n}+ \ddiv(\mff(\mx, u_n)) &=\triangle u_n+\partial_{t}\triangle u_n, \ \ (t,x)\in [0,T)\times B_{n-1}\\
u_n(0,\mx)&=u_0(\mx), \ \ \mx\in B_{n-1}
\end{split}
\end{equation*} Extending $u_n$ by zero outside of $B_n$, we thereby obtain a sequence $(u_n)$ that is bounded in $L^2([0,T]\times \R^d)$, 
and locally bounded in $H^1((0,T)\times \R^d)$. Thus, $(u_n)$ admits a subsequence that locally converges towards $u\in L^2([0,T]\times \R^d)$. It is clear that the function $u$ is the weak solution to \eqref{main-eq}, \eqref{id-nonpert} and thus it must belong at least to $H^1((0,T)\times \R^d)$. 

Moreover, the proof of Theorem \ref{estimates-1} with $v=u$ is a classical example of a bootstrapping procedure, which enables 
us to conclude that $u$ actually belongs to each  $H^k((0,T)\times \R^d)$ ($k\in \N$), hence is
smooth (due to smoothness of $u_0$ and $\mff$).

Now we turn to the proof of uniqueness. Assume that $u_1$ and $u_2$ are solutions to \eqref{main-eq}, supplemented with the initial conditions $u_1\big{|}_{t=0}=u_{10}$ and $u_2\big{|}_{t=0}=u_{20}$. After applying the procedure from Theorem \ref{estimates-1} we conclude that 

\begin{align}
\label{u-1}
\hat{u}_1(t,\mxi)-\hat{u}_2(t,\mxi)=e^{\frac{-t|\mxi |^2}{1+|\mxi|^2}} \left( \hat{u}_{10}(\mxi)-\hat{u}_{20}(\mxi) -\int_0^t \frac{i\mxi \cdot (\hat{\mff}_1-\hat{\mff}_2)}{1+|\mxi|^2}e^{\frac{t' |\mxi |^2}{1+|\mxi|^2}}dt'\right),
\end{align}  where
$$
\hat{\mff}_1-\hat{\mff}_2={\cal F}(\mff(\mx,u_1(t,\mx))-\mff(\mx,u_2(t,\mx))).
$$ 
%If we apply here the mean value theorem for derivatives and the connection between the Fourier transform and the convolution, we reach to
%\begin{align}
%\label{u-2}
%&\frac{\mxi\cdot (\hat{\mff}_1-\hat{\mff}_2)}{1+|\mxi|^2}={\cal F}\left(\sum\limits_{j=1}^d\left( \frac{\xi_j f_j(\mx,u_1)}{1+|\mxi|^2}-\frac{\xi_j f_j(\mx,u_2)}{1+|\mxi|^2}\right) \right) =\\&
%{\cal F}\left(F(\mxi,\mx,\tilde{u}(t,\mx)) (u_1(t,\mx)-u_2(t,\mx))\right))= \nonumber
%\\&{\cal F}\left(F(\mxi,\mx,\tilde{u}(t,\mx))\right)\star {\cal F}\left( (u_1(t,\mx)-u_2(t,\mx))\right))
%\nonumber
%\end{align}where 
%$$
%F(\mxi,\mx,\lambda)=\sum\limits_{j=1}^d \frac{\xi_j f_j(\mx,\lambda)}{1+|\mxi|^2}
%$$ and $\tilde{u}$ is a function obtained after application of the mean value theorem on the difference $F(\mxi,\mx,u_1)-F(\mxi,\mx,u_2) $. Remark that it is a smooth function since $F(\mxi,\mx,u_1)-F(\mxi,\mx,u_2)$ is smooth. Moreover, from the Young inequality for convolutions, we have from \eqref{u-2}

Let us first estimate 
\begin{align}
\label{u-3}
\left\|\frac{\mxi\cdot (\hat{\mff}_1-\hat{\mff}_2)}{1+|\mxi|^2}\right\|_{L^2(\R^d)} &\leq  \| \hat{\mff}_1-\hat{\mff}_2\|_{L^2(\R^d)} 
=  \| \mff_1-\mff_2\|_{L^2(\R^d)}  \\
&\leq  \|\pa_\lambda \mff\|_{\infty}\| \hat{u}_1-\hat{u}_2\|_{L^2(\R^d)}. \nonumber
\end{align}
Taking the $L^2(\R^d)$-norm of \eqref{u-1} we get, due to \eqref{u-3}: 
\begin{align}
\label{u-4}
&\|\hat{u}_1(t,\cdot)-\hat{u}_2(t,\cdot)\|_{L^2(\R^d)} \leq \\ & \|\hat{u}_{10}-\hat{u}_{20}\|_{L^2(\R^d)} + \|\pa_\lambda \mff \|_{\infty}  \int_0^t\|\hat{u}_1(t',\cdot)-\hat{u}_2(t',\cdot)\|_{L^2(\R^d)} e^{t'} dt'.
\nonumber
\end{align} 
Applying Gronwall's inequality, we arrive at
\begin{equation*}
\|\hat{u}_1(t,\cdot)-\hat{u}_2(t,\cdot)\|_{L^2(\R^d)} \leq \|\hat{u}_{10}-\hat{u}_{20}\|_{L^2(\R^d)} 
e^{(e^t-1)\|\pa_\lambda \mff\|_{\infty}}.
\end{equation*} 
From here, we see that  $u_{10}=u_{20}$ entails $u_1=u_2$. This concludes the proof.
\end{proof}

\section{Vanishing capillarity limit}

In this section, we inspect the vanishing capillarity limit of \eqref{vc-1}. By Theorem \ref{unique_sol_par}, under the assumptions (C1)--(C3) from Section \ref{intro}, the equation \eqref{vc-1} with initial data
\eqref{id-pert} (satisfying \eqref{cond}) possesses a unique solution $u_\eps$ in $L^2((0,T)\times \R^d)\cap C^\infty((0,T)\times \R^d)$.
As announced in Section \ref{intro}, our main result then is as follows:

\begin{theorem} Under the above assumptions, if
$(F_1,\dots, F_d)=\partial_\lambda \mff$ lies in $L^{\bar{p}'}(\R^{d+1})$
for some $\bar{p}'\in \big(1,2+\frac{4}{d}\big)$ then: 	
\label{main-theorem}
\begin{itemize}
	
	\item[(i)] If $\delta=o(\eps^2)$  and $n(\eps)$ satisfies \eqref{neps} 
	then the family $(u_\eps)$ contains a strongly convergent subsequence in $L^1_{loc}(\R^+\times \R^d)$ under the non-degeneracy condition 
	given in Definition \ref{def-non-deg} below.
	
	\item[(ii)] If $\delta={\cal O}(\eps^2)$ and $\frac{\sqrt{\delta}}{\sqrt{\eps} n(\eps)^{d/2+2}}\to 0$ as $\eps\to 0$, and the $\lambda$-derivative of the flux $\pa_\lambda \mff$ is bounded in addition to condition
	(C3), then the family $(u_\eps)$ contains a strongly convergent subsequence in $L^1_{loc}(\R^+\times \R^d)$ under the non-degeneracy condition;
	
	\item[(iii)] If $\delta=o(\eps^2)$ and $\frac{\sqrt{\delta}}{\sqrt{\eps} n(\eps)^{d/2+2}}\to 0$ as $\eps\to 0$, and if the flux $\mff\in C^1(\R^d\times \R)$, then the family $(u_\eps)$ converges strongly in $L^1_{loc}(\R^+\times \R^d)$ towards the entropy solution to \eqref{cl-1}, \eqref{id}.
\end{itemize}
\end{theorem}

Below we give a complete proof of (i). Although non-trivial, the other two claims are more standard and we shall only comment on the necessary 
modifications of the argument for establishing (i). 

To proceed, let us first introduce the non-degeneracy condition. Such a condition is standard in many results using velocity averaging lemmas 
(cf., e.g., \cite{LM5, LPT, Pan, PS, TT}).

\begin{definition}
\label{def-non-deg}
The flux 
appearing in \eqref{cl-1} is called non-degenerate if for the 
function ${\bf F}=(F_1,\dots, F_d)=\partial_\lambda \mff$, where $\mff=\mff(\mx,\lambda)$ we have: for almost every $(t,\mx)\in \R_+^{d}$ and every $\mxi\in S^{d}$, where $S^d$ is the unit sphere in $\R^{d+1}$ the mapping

\begin{equation}
\label{non-deg} \lambda \mapsto  \left( \xi_0+
\sum\limits_{k=1}^d F_k(\mx,\lambda) \xi_k \right),
\end{equation} is not zero on any set of positive measure.
\end{definition} Under these assumptions, we proved the following assertion in \cite[Theorem 3.4]{LM5}.

\begin{theorem}
\label{va}
 Assume that $h_n\rightharpoonup h$  weakly in
${L}_{loc}^{s}(\R^{d+1})$ for some $s\geq 2$, where
$h_n$ are weak solutions to 

\begin{equation}
\label{lin-eq}
\begin{split}
\sum\limits_{k=1}^d\pa_{x_k}
\left(F_k(\mx,\lambda) h_n(\mx,\lambda)\right)=\pa_\lambda^\kappa
G_n(\mx,\lambda)\,.
\end{split}
\end{equation} 
Here, $\kappa\in \N_0$ and $\lambda \in \R$.  Let $\bar{p}\in (1,s)$ 
and let $\frac{1}{\bar{p}}+\frac{1}{\bar{p}'}=1$. We assume
\begin{itemize}

\item [\bf a)]
$F_k\in { L}^{\bar p'}(\R^{d+1})$ for $k=1,\dots,d, \ $

\item [\bf b)] The sequence $(G_n)$ is strongly precompact in the
space 
${ W}_{loc}^{-1,r}(\R^{d+1})$, where $r>1$ satisfies the relation 
$1+\frac{1}{s}=\frac{1}{\bar{p}}+\frac{1}{r}$.

\end{itemize} Finally, assume that the non-degeneracy condition \eqref{non-deg} is satisfied.

Then, for any $\rho\in L^2_c(\R)$, there exists a subsequence $(h_{n_k})$ of $(h_n)$ such that
\begin{equation}
\label{va_result}
\int_{\R}\rho(\lambda) h_{n_k}(\mx,\lambda)d\lambda \to  \int_{\R}\rho(\lambda) h(\mx,\lambda)d\lambda \ \ \text{ strongly in
$L^1_{loc}(\R^d)$}
\end{equation}
as $k\to\infty$.
\end{theorem}

Now, we proceed to the estimates sufficient to apply the mentioned functional analytic tools.

\begin{lemma}
\label{osnov}
The solution $u_{\eps}$ of \eqref{vc-1} satisfies the following a priori estimates: 
\begin{align}
\label{prva}
\|u_{\eps}(t,\,.\,)\|_{L^{2}(\R^{d})}\leq \|u_{0}^{\eps}\|_{L^{2}(\R^{d})}+\sqrt{\delta}\|\nabla u^\eps_{0}\|_{L^{2}(\R^{d})}+\sqrt{2 t} \| \ddiv \mff_\eps(\mx,\lambda) \|^{\frac{1}{2}}_{L^{1}(\R^{d}\times \R)},
\end{align}
\begin{equation}
\label{druga}
\begin{split}
\sqrt{2\eps}\|\nabla u_{\eps}\|_{L^{2}([0,T],L^{2}(\R^{d}))} \leq \|u_{0}^{\eps}\|_{L^{2}(\R^{d})} +\sqrt{\delta} & 
\|\nabla u_{0}^{\eps}\|_{L^{2}(\R^{d})}
\\ &+\sqrt{2 T} \| \ddiv \mff_\eps(\mx,\lambda) \|^{\frac{1}{2}}_{L^{1}(\R^{d}\times \R)},
\end{split}
\end{equation}
\begin{equation}
\label{treca}
\begin{split}
\sqrt{\delta} \|\nabla u_{\eps}(t,\,.\,) \|_{L^{2}(\R^{d})}\leq \|u_{0}^{\eps}\|_{L^{2}(\R^{d})}+\sqrt{\delta}  \|\nabla & u_{0}^{\eps}\|_{L^{2}(\R^{d})}\\
&+\sqrt{2 t} \| \ddiv \mff_\eps(\mx,\lambda) \|^{\frac{1}{2}}_{L^{1}(\R^{d}\times \R)}
\end{split}
\end{equation}
\begin{equation}
\label{cetvrta}
\begin{split}
\eps \delta \|\nabla \pa_{x_i} u_{\eps} & \|^2_{L^{2}([0,T],L^{2}(\R^{d}))} \leq 
2\delta \|\pa_{x_i}u_{0}^{\eps}\|^2_{L^{2}(\R^{d})}+2\delta^2 \|\nabla \pa_{x_i} u_{0}^{\eps}\|^2_{L^{2}(\R^{d})}  \\  
&+\frac{d\delta}{\eps^2} \|\pa_\lambda \mff_\eps\|^2_{\infty} ( \|u_{0}^{\eps}\|^2_{L^{2}(\R^{d})}+\delta\|\nabla u_{0}^{\eps}\|^2_{L^{2}(\R^{d})} \\
& +2 T \| \ddiv \mff_\eps(\mx,\lambda) \|_{L^{1}(\R^{d}\times \R)}) +\frac{4  T \delta }{\eps} \| \ddiv \mff^\eps(\mx, \lambda)\|^2_{L^2(\R^d, L^\infty (\R))} 
\end{split}
\end{equation}
\begin{equation}
\label{peta}
\begin{split}
 \delta \|\nabla & \pa_t u_\eps\|_{L^2([0,T],L^2(\R^d))} 
 \leq  \frac {\sqrt{\delta}  \sqrt{d}}{2\sqrt{\eps}} \|\pa_\lambda \mff_\eps\|_{\infty}  \|u^\eps_{0}\|_{L^{2}(\R^{d})}\\  
 &\hspace*{-2em}+
 \sqrt{\delta}  \left(\frac{\sqrt{\delta}\sqrt{d}}{2\sqrt{\eps}} \|\pa_\lambda \mff_\eps\|_{\infty}
 +\frac{\sqrt{\eps}}{\sqrt{2}}\right)\|\nabla u^\eps_{0}\|_{L^{2}(\R^{d})} \\
 &\hspace*{-2em} +\frac{\sqrt{\delta}  \sqrt{dT}}{\sqrt{2\eps}}  \|\pa_\lambda \mff_\eps\|_{\infty}
 \|\ddiv \mff_\eps \|_{L^1(\R^d\times \R)}^{\frac{1}{2}}+\sqrt{\delta}   \sqrt{T} 
 \|\ddiv \mff_\eps \|_{L^2(\R^d,L^\infty(\R))}.
\end{split}
\end{equation}

\end{lemma}

\begin{proof}
Let us consider a family $(u_{\eps})$ of smooth solutions to \eqref{vc-1}.
Multiplying the equation by $\eta'(u_\eps)$ (with $\eta(u_\eps)=u_\eps^{2})$, we have
\begin{align}
\label{-2}
&\pa_t(u_\eps^{2})+2\sum_{j=1}^{d}q^\eps_{j}(\mx,u_{\eps})_{x_j}+2 \int_0^{u_{\eps}} \ddiv \mff_\eps(\mx,\lambda) d\lambda 
\\&=2\eps\sum_{j=1}^{n}(u_\eps \pa_{x_j}u_{\eps})_{x_j}-2\eps|\nabla u_\eps|^{2}+ \delta(\triangle\eta(u_\eps))_{t}-2\delta \sum_{j=1}^{n}\eta(\pa_{x_j}u_{\eps})_t-2\delta\triangle u_{\eps} \pa_t u_{\eps},
\nonumber
\end{align}
where $\bf q_\eps$ is defined by (recall that $\mff_\eps=(f_1^\eps,\dots,f_d^\eps)$) 
$$\pa_\lambda{\bf q}_\eps(\mx,\lambda)=(\pa_\lambda F^\eps_1(\mx,\lambda), \dots, \pa_\lambda F^\eps_d(\mx,\lambda))
=(\lambda \pa_\lambda f^\eps_1(\mx,\lambda),\dots, \lambda \pa_\lambda f^\eps_d(\mx,\lambda)),
$$ 
and normalized by the  condition $q^\eps_{j}(\mx,0)=0$, $j=1,...,d$.
Integrating over $\mx\in \R^d$ and  $0\leq t\leq T$, and using integration by parts we get 
\begin{equation*}
\begin{split}
\int_{\R^{d}}(|u_\eps(t) &|^{2}  -|u^\eps_{0}|^{2})dx+2 \int_0^t \int_{\R^d}\int_0^{u_{\eps}(\mx,t')} \ddiv \mff_\eps(\mx,\lambda) \,d\lambda d\mx dt'
\\&=-2\eps\int_{0}^{t}\int_{\R^{d}}|\nabla u_\eps|^{2}d\mx dt'+\delta\int_{\R^{d}}\triangle (\eta(u_\eps(t,\mx))-\eta(u^\eps_0(\mx))d\mx\\
&-2\delta \int_{\R^{d}}\sum_{j=1}^{d}(|\pa_{x_j} u_{\eps}(t,\mx)|^{2}-|\pa_{x_j} u^\eps_{0}(\mx)|^{2}(\mx))d\mx-2\delta\int_{0}^{t}\int_{\R^{d}}\triangle u_{\eps} \pa_t u_{\eps}d\mx dt'\\
&=-2\eps\int_{0}^{t}\int_{\R^{d}}|\nabla u_\eps|^{2}d\mx dt'-2\delta \int_{\R^{d}}\sum_{j=1}^{d}(|\pa_{x_j}u_{\eps}(t,\mx)|^{2}-|\pa_{x_j}u^\eps_{0}|^{2}(\mx))d\mx\\
&+\delta\sum_{j=1}^{d}\int_{\R^{d}}\int_{0}^{t}\eta(\pa_{x_j}u_{\eps})_{t}dtdx=-2\eps\int_{0}^{t}\int_{\R^{d}}|\nabla u_\eps|^{2}\,d\mx dt'\\
&-\delta \int_{\R^{d}}\sum_{j=1}^{d}|\pa_{x_j} u_{\eps}(t,\mx)|^{2}-|\pa_{x_j} u^\eps_{0}(\mx)|^{2}(\mx))d\mx
\end{split}
\end{equation*} i.e.,
\begin{equation}
\label{main-eq2}
\begin{split}
\int_{\R^{d}}|u_\eps(t)|^{2}dx+2\eps\int_{0}^{t} &\int_{\R^{d}}|\nabla u_\eps|^{2}dxdt'+\delta \sum_{j=1}^{d}\int_{\R^{d}}|\pa_{x_j}u_{\eps}|^{2}d\mx\\
&=\int_{\R^{d}}|u_{0}^\eps|^{2}dx+\delta \sum_{j=1}^{d}\int_{\R^{d}}|\pa_{x_j} u^\eps_{0}(\mx)|^{2}(x)dx\\
&\hspace*{3em}-2 \int_0^t \int_{\R^d}\int_0^{u_{\eps}(\mx,t')} \ddiv \mff_\eps(\mx,\lambda) d\lambda d\mx dt'.
\end{split}
\end{equation}
The relations \eqref{prva}, \eqref{druga} and \eqref{treca} are easy consequences of \eqref{main-eq2}.

%%%%%%%%%%%%%% derivatives %%%%%%%%%%%%%%

To show \eqref{cetvrta}, let us now differentiate equation \eqref{vc-1} with respect to $x_{i}$ and then multiply it by $2\pa_{x_i}u_{\eps}.$ Since the right hand side of \eqref{vc-1} is linear in $u_\eps$, the calculation for that side will remain the same and
after integrating the equation over $[0,t]\times \R^d$ and applying integration by parts, we get
\begin{equation*}
\begin{split}
&\int_{\R^{d}}(|\pa_{x_i}u_{\eps}(t,\mx)|^{2}-|\pa_{x_i}u^\eps_{0}(\mx)|^{2})d\mx-2\sum_{k=1}^{d}\int_{0}^{t}\int_{\R^{d}} \pa_\lambda f^\eps_{k}(\mx, u_\eps) \, \pa_{x_k} u_\eps \, \pa_{x_i}^2 u_{\eps}d\mx dt\\
&-2\int_0^t \int_{\R^d} \pa_{x_i}^2 u_\eps \, \ddiv \mff^\eps(\mx, \lambda)\big{|}_{\lambda=u_\eps} d\mx dt =-2\eps\int_{0}^{t}\int_{\R^{d}}|\nabla \pa_{x_i}u_\eps|^{2}dxds\\
&-\delta \int_{\R^{d}}\sum_{j=1}^{d}(|\pa_{x_{i}x_j}u_\eps(t,\mx )|^{2}-|\pa_{x_{i}x_j}u^\eps_{0}|^{2})d\mx,
\end{split}
\end{equation*}
i.e., 

\begin{equation}
\begin{split}
&\int_{\R^{d}}|\pa_{x_i} u_{\eps}(t,\mx)|^{2} d\mx+2\eps\int_{0}^{t}\int_{\R^{d}}|\nabla \pa_{x_i}u_\eps |^{2}d\mx dt'+\delta \int_{\R^{d}} |\nabla \pa_{x_{i}} u_\eps(t,\mx)|^{2}d\mx\\
&=\int_{\R^{d}}| \pa_{x_i}u^\eps_{0}(\mx)|^{2}d\mx+\delta \int_{\R^{d}} |\nabla \pa_{x_{i}}u^\eps_{0}(\mx)|^{2}d\mx 
+2 \int_0^t \int_{\R^d} \pa_{x_i}^2 u_\eps \, \ddiv \mff^\eps (\mx, \lambda)\big{|}_{\lambda=u_\eps} d\mx dt' \\ &\hspace*{6em}+2\sum_{k=1}^{d}\int_{0}^{t}\int_{\R^{d}}\pa_\lambda f^\eps_{k}(\mx,u_\eps) \pa_{x_k}u_\eps  \pa_{x_i}^2 u_{\eps}dxdt' \\
&\leq \|\pa_{x_i}u_{0}^{\eps}\|^2_{L^{2}(\R^{d})}+\delta \|\nabla \pa_{x_i} u_{0}^{\eps}\|^2_{L^{2}(\R^{d})}   
+ \frac{\eps}{2} \int_{0}^{t}\int_{\R^{d}} |\pa_{x_{i}}^2 u_{\eps}|^2 d\mx dt'
\\
&  \hspace*{6em} +\frac{2}{\eps} \| \ddiv \mff^\eps(\mx, \lambda)\big{|}_{\lambda=u_\eps}\|^2_{L^2([0,T]\times \R^d)}
+\frac{d \|\pa_\lambda \mff_\eps\|^2_{\infty}}{\eps} \int_{0}^{t}\int_{\R^{d}}|\nabla u_{\eps}|^2 d\mx dt'\\
 & \hspace*{12em} +\eps  \int_{0}^{t}\int_{\R^{d}} |\pa_{x_{i}}^2 u_{\eps}|^2 d\mx dt'
 \\ 
& \leq  \|\pa_{x_i}u_{0}^{\eps}\|^2_{L^{2}(\R^{d})}+\delta \|\nabla \pa_{x_i} u_{0}^{\eps}\|^2_{L^{2}(\R^{d})}  +\frac{d \|\pa_\lambda \mff_\eps\|^2_{\infty}}{\eps} \|\nabla u_{\eps}\|^2_{L^2([0,T],L^2(\R^d))}  \\ 
&\hspace*{5em}+ \frac{3\eps}{2}  \int_{0}^{t}\int_{\R^{d}} |\nabla \pa_{x_{i}} u_{\eps}|^2 d\mx dt' +\frac{2}{\eps} \| \ddiv \mff^\eps(\mx, \lambda)\big{|}_{\lambda=u_\eps}\|^2_{L^2([0,T]\times \R^d)}
\label{est-d-2} 
\end{split}
\end{equation} 
where $\|\pa_\lambda \mff_\eps\|_{\infty}=\max_{1\leq j\leq n}\|\pa_{\lambda} f^\eps_{j}\|_{L^\infty(\R^{d+1})}$.  
Multiplying the above inequality by $\delta$, we have

%Further,
%
%\begin{equation}
%\begin{split}
%&\int_{\R^{d}}u_{0x_i}^{2}(t,x)dx+\delta \sum_{j=1}^{d}\int_{\R^{d}}u_{0x_{i}x_j}^{2}(x)dx+2\|f'\|_{\infty}^{2}\sum_{j=1}^{d}\int_{0}^{t}\int_{\R^{d}}|u_{x_j}||u_{x_{i}x_{i}}|dxdt\\
%&\leq \int_{\R^{d}}u_{0x_i}^{2}(t,x)dx+\delta \sum_{j=1}^{d}\int_{\R^{d}}u_{0x_{i}x_j}^{2}(x)dx+\frac{d\|f'\|_{\infty}^{2}}{\eps}\sum_{j=1}^{d}\int_{0}^{t}\int_{\R^{d}}u_{x_{j}}^{2}dx+\eps\int_{0}^{t}\int_{\R^{d}}u_{x_{i}x_{i}}^{2}dxdt.
%\end{split}
%\end{equation}

\begin{equation*}
\begin{split}
&\delta \int_{\R^{d}}|\pa_{x_i}u_{\eps}(t,\mx)|^{2}d\mx +\frac{\eps \delta}{2} 
\int_{0}^{t}\int_{\R^{d}}|\nabla \pa_{x_i} u_{\eps}|^{2}d\mx dt + \delta^2 \int_{\R^{d}}|\nabla \pa_{x_{i}}u_{\eps}(t,\mx)|^{2}d\mx\\
& \leq \delta \|\pa_{x_i}u_{0}^{\eps}\|^2_{L^{2}(\R^{d})}+\delta^2 \|\nabla \pa_{x_i} u_{0}^{\eps}\|^2_{L^{2}(\R^{d})}  +
\frac{d \delta}{\eps} \|\pa_\lambda \mff_\eps\|^2_{\infty} \|\nabla u_{\eps}\|^2_{L^2([0,T],L^2(\R^d))} \\&
+\frac{2\delta }{\eps} \| \ddiv \mff^\eps(\mx, \lambda)\big{|}_{\lambda=u_\eps}\|^2_{L^2([0,T]\times \R^d)} \\
& \leq \delta \|\pa_{x_i}u_{0}^{\eps}\|^2_{L^{2}(\R^{d})}+\delta^2 \|\nabla \pa_{x_i} u_{0}^{\eps}\|^2_{L^{2}(\R^{d})} 
+\frac{2T \delta }{\eps} \| \ddiv \mff^\eps(\mx, \lambda)\|^2_{L^2(\R^d, L^\infty (\R))} \\
&+\frac{d\delta}{2\eps^2} 
\|\pa_\lambda \mff_\eps\|^2_{\infty} ( \|u_{0}^{\eps}\|^2_{L^{2}(\R^{d})}+\delta\|\nabla u_{0}^{\eps}\|^2_{L^{2}(\R^{d})}
+2 T \| \ddiv \mff_\eps(\mx,\lambda) \|_{L^{1}(\R^{d}\times \R)}),
\end{split}
\end{equation*} where in the last inequality we used equation \eqref{main-eq2}. From here, \eqref{cetvrta} immediately follows.

We next aim to prove inequality \eqref{peta}. To this end we multiply equation \eqref{vc-1} by $\pa_t u_{\eps}$ and 
and then proceed similarly to the above. Namely,

\begin{equation}
\label{main-eq11}
\begin{split}
(\pa_t u_{\eps})^{2}+\sum_{j=1}^{d} \pa_t u_{\eps} \, \pa_\lambda f_j^\eps(\mx,u_\eps) \pa_{x_j}u_{\eps} + \pa_t u_{\eps} \, &\ddiv \mff^\eps(\mx,\lambda)|_{\lambda=u_\eps} \\
&=\eps \pa_t u_{\eps}\triangle u_\eps+\delta \pa_t u_{\eps}\, \triangle \pa_t u_{\eps},
\end{split}
\end{equation} 
and we integrate equation \eqref{main-eq11} with respect to time and space.

Integration by part gives

\begin{align*}
\int_{0}^{t}\int_{\R^{d}}|\pa_t u_{\eps}|^{2}&d\mx dt'+\sum_{j=1}^{d}\int_{0}^{t}\int_{\R^{d}}\pa_t u_{\eps}\pa_\lambda f_j^\eps (\mx, u_\eps) \pa_{x_j} u_{\eps}d\mx dt' \\
&+\int_{0}^{t}\int_{\R^{d}} \pa_t u_{\eps} \, \ddiv \mff^\eps(\mx,\lambda)|_{\lambda=u_\eps} d\mx d\lambda \\
&=\eps\int_{0}^{t}\int_{\R^{d}}\triangle u_\eps \, \pa_t u_{\eps} d\mx dt'+\delta\int_{0}^{t}\int_{\R^{d}}\pa_t u_{{\eps}}\triangle \pa_t u_{\eps}d\mx dt'
\nonumber \\
&=-\sum_{j=1}^{d}\int_{0}^{t}\eps\int_{\R^{d}}\pa_{x_j} u_{\eps}\pa_{tx_j}u_{\eps} d\mx dt'-\delta\int_{0}^{t}\int_{\R^{d}}|\nabla \pa_t u_{\eps}|^{2}d\mx dt',
\nonumber
\end{align*} i.e. 

\begin{equation}
\begin{split}
\label{est-51}
&\int_{0}^{t}\int_{\R^{d}} |\pa_t u_{\eps}|^{2}d\mx dt'+\delta\int_{0}^{t}\int_{\R^{d}}|\nabla \pa_t u_{\eps}|^{2}d\mx dt'=-\eps\sum_{j=1}^{d}\int_{0}^{t}\int_{\R^{d}}\pa_{x_j} u_{\eps} \pa_{tx_j} u_{\eps}d\mx dt'\\
&- \sum_{j=1}^{d}\int_{0}^{t}\int_{\R^{d}}\left(\pa_t u_{\eps}\pa_\lambda f_j^\eps (\mx, u_\eps) \pa_{x_j} u_{\eps}
- \pa_t u_{\eps} \, \pa_{x_j} f_j^\eps(\mx,\lambda)|_{\lambda=u_\eps} \right) d\mx dt'    \\
&\leq \| \pa_\lambda \mff \|_{\infty}\sum_{j=1}^{d}\int_{0}^{t}\int_{\R^{d}}|\pa_t u_{\eps}| \, |\pa_{x_j} u_{\eps}|d\mx dt' \\
& \hspace*{1em}+ \| \pa_t u_\eps \|_{L^2([0,T]\times \R^d)} \sqrt{T} \|\ddiv \mff_\eps \|_{L^2(\R^d,L^\infty (\R))}  -\frac{\eps}{2}\sum_{j=1}^{d}\int_{0}^{t}\int_{\R^{d}}(|\pa_{x_j} u_{\eps}|^{2})_t d\mx dt'   \\
&\leq \frac{3}{4}\int_{0}^{t}\int_{\R^{d}}|\pa_t u_{\eps}|^{2} d\mx dt'+\frac{d \|\pa_\lambda \mff_\eps\|_{\infty}^{2}}{2}\int_{0}^{t}\int_{\R^{d}}|\nabla u_\eps|^{2}d\mx dt'  \\
&+T \|\ddiv \mff_\eps \|^2_{L^2(\R^d, L^\infty(\R))}
-\frac{\eps}{2}\int_{\R^{d}}|\nabla u_\eps(t)|^{2}dx+\frac{\eps}{2}\int_{\R^{d}}|\nabla u^\eps_{0}|^{2}d\mx .
 \end{split}
\end{equation}
 
By multiplying \eqref{est-51} by $\delta$
and taking into account \eqref{druga}, we conclude
\begin{align}
&\frac{\delta}{4}\int_{0}^{t}\int_{\R^{d}}|\pa_t u_{\eps}|^{2} d\mx dt'+\frac{\eps \delta}{2}\int_{\R^{d}}|\nabla u_\eps (t)|^{2}d\mx+\delta^{2} \int_{0}^{t}\int_{\R^{d}}|\nabla \pa_t u_{\eps}|^{2}d\mx dt' \\
& \nonumber \leq \frac{\delta d\| \pa_\lambda \mff_\eps\|_{\infty}^{2}}{2}\int_{0}^{t}\int_{\R^{d}}|\nabla u_\eps|^{2}dxdt'+\frac{\delta\eps}{2}\int_{\R^{d}}|\nabla u^\eps_{0}|^{2} d\mx+\delta T 
\|\ddiv \mff_\eps \|^2_{L^2(\R^d, L^\infty (\R))} \nonumber \\
&\leq \frac {\delta d\|\pa_\lambda \mff_\eps\|_{\infty}^{2}}{4\eps}
\|u^\eps_{0}\|_{L^{2}(\R^{d})}^{2}+
\left(\frac {\delta^2 d\|\pa_\lambda \mff_\eps\|_{\infty}^{2}}{4\eps}+\frac{\delta  \eps}{2}\right)\|\nabla u^\eps_{0}\|_{L^{2}(\R^{d})}^{2} \nonumber \\
& +\frac {\delta d\|\pa_\lambda \mff_\eps\|_{\infty}^{2}T}{2\eps}  \|\ddiv \mff_\eps \|_{L^1(\R^d\times \R)}+\delta T \|\ddiv \mff_\eps \|^2_{L^2(\R^d,L^\infty(\R))},
\nonumber
\end{align}which gives the inequality \eqref{peta}.
\end{proof}

%%%%%%%%%%%Kinetic%%%%%%%%%%%%%%

We shall now derive the kinetic formulation for \eqref{vc-1} and then show that we can apply the velocity averaging result on this
form of the equation. 

\begin{lemma} {\bf (Kinetic formulation of \eqref{vc-1})}
Assume that the function $u_\eps$ satisfies \eqref{vc-1}. Then, the function $h_\eps(t,\mx,\lambda)={\rm sgn}(u_\eps(t,\mx)-\lambda)$ solves a linear PDE of the form
\begin{align}
\label{kinetic}
&\pa_t h_\eps +\ddiv (h_\eps \pa_\lambda \mff_\eps (\mx,\lambda) )-\pa_\lambda  (h_\eps \ddiv \mff_\eps(\mx,\lambda))=\pa_\lambda \ddiv (G_1^\eps)+ \pa^2_\lambda G^\eps_2,
\end{align} where 
\begin{itemize}

\item[(i)] 
%\sout{$G_1^\eps \to 0$ as $\eps\to 0$ in $L^2(\R^+\times \R^d\times \R)$ and thus} 
%\melanietodo{we never show this, but also don't need it} 
$\ddiv G_1^\eps \to 0$ in $W^{-1,r}_{loc}(\R^+\times \R^d\times \R)$ for any $r\in [1,2)$; 

\item[(ii)] $(G^\eps_2)$ is bounded in the space of measures ${\cal M}_{loc}(\R^+\times \R^d \times \R)$ and thus strongly precompact in $W_{loc}^{-1,r}(\R^+\times \R^d \times \R)$ for any $r\in [1,\frac{d+2}{d+1})$.

\end{itemize}

\end{lemma}
\begin{proof}
Before we start, notice that we require merely local convergence (see conditions (i) and (ii) of the lemma), 
which means that we can always choose $\eps$ small enough so that $K_\eps \equiv 1$ on the fixed relatively compact 
set on which we derive the estimates.

Therefore, it is enough to prove the lemma under the simplifying assumption
$$
\mff_\eps=\mff\star\omega_{n(\eps)}.
$$ 
Also, in order to avoid proliferation of symbols, we will notationally adhere to integration over the entire space $\R^d\times \R$,
with the implicit understanding we are in fact integrating over the support of the corresponding test function (see \eqref{20}).

Then, note that for any $\eta \in C^3(\R)$ such that $\eta'\in C^2_c(\R)$ 
%supported in $(-M,M)$, 
we have (recall that $h_\eps(t,\mx,\lambda)={\rm sgn}(u_\eps(t,\mx)-\lambda)$) 
\begin{equation}
\label{kinetic-1}
\eta(u_\eps)  = \frac{1}{2} \int_\R \eta'(\lambda) h_\eps(t,\mx,\lambda)  d\lambda+\frac{\eta(-R)+\eta(R)}{2}
\end{equation} for $R>0$ such that ${\rm supp}(\eta')\subset [-R,R]$. Moreover,  for any function $f\in L^1(\R)$ 
\begin{equation}
\label{kinetic-2}
\int_{-\infty}^{u_\eps} f(\lambda) d\lambda=\frac{1}{2}\int_\R f(\lambda )  h_\eps(t,\mx,\lambda)  d\lambda +\frac{1}{2}\int_\R f(\lambda) d\lambda.
\end{equation}  To proceed, we multiply equation \eqref{vc-1} by $\eta'(u_\eps)$. We get 

\begin{align}
\label{kinetic-new}
&\partial_{t} \eta(u^{\eps})+\ddiv \int_{-\infty}^{u_\eps} \eta'(\lambda) \pa_\lambda \mff_\eps(\mx,\lambda) d\lambda+\int_{-\infty}^{u_\eps} \eta''(\lambda)\ddiv \mff_\eps(\mx,\lambda) d\lambda\\&
=\sum_{j=1}^{d}\partial_{x_j}\left(\eps\eta'(u_{\eps})\pa_{x_j}u_{\eps}+\delta\eta'(u_{\eps})\pa_{t,x_{j}}u_{\eps}\right) \nonumber \\
&-\eps \eta''(u_{\eps})\sum_{j=1}^{d}|\pa_{{x_j}}u_\eps|^{2}-\delta\eta''(u_\eps)\sum_{j=1}^{d}\pa_{x_{j}}u_{\eps}\pa_{t,x_{j}}u_{\eps}.
\nonumber 
\end{align} Now, we apply \eqref{kinetic-1} and \eqref{kinetic-2} and rewrite the latter equation in the form 

\begin{align*}
&\frac{1}{2}
\partial_{t} \int_{\R} \eta'(\lambda) h_\eps d\lambda +\ddiv \int_\R  \eta'(\lambda) 
\pa_\lambda \mff_\eps(\mx,\lambda) h_\eps d\lambda+\int_\R \eta''(\lambda)\ddiv \mff_\eps(\mx,\lambda) h_\eps d\lambda \\&
=\sum_{j=1}^{d}\partial_{x_j}\left(\eps \int_\R \eta''(\lambda) h_\eps d\lambda \pa_{x_j}u_{\eps}+\delta\int_\R \eta''(\lambda) h_\eps d\lambda \pa_{t,x_{j}}u_{\eps}\right)\\
&-\eps \int_\R \eta'''(\lambda) h_\eps d\lambda\sum_{j=1}^{d}|\pa_{{x_j}}u_\eps|^{2}-\delta\int_\R \eta'''(\lambda) h_\eps d\lambda\sum_{j=1}^{d}\pa_{x_{j}}u_{\eps}\pa_{t,x_{j}}u_{\eps}.
\end{align*} 
If we rewrite the latter in the variational formulation, we get for every test function $\phi \in C_c^2(\R^+\times \R^d)$
\begin{align}
\label{20}
&-\int_{\R^+\times \R^d \times \R} \left( \frac{1}{2}\partial_{t}\phi (t,\mx)  \eta'(\lambda) +  \nabla \phi \eta'(\lambda) \cdot  \pa_\lambda \mff_\eps(\mx,\lambda)\right) h_\eps dt d\mx d\lambda\\
&+\int_{\R^+\times \R^d \times \R}  \eta''(\lambda) \, \phi \, \ddiv \mff_\eps (\mx ,\lambda) h_\eps dt d\mx  d\lambda
\nonumber
\\&
=-\sum_{j=1}^{d}\int_{\R+\times \R^d \times \R}h_\eps \partial_{x_j}\phi\eta''(\lambda)  \left(\eps  \pa_{x_j}u_{\eps}+\delta \pa_{t,x_{j}}u_{\eps}\right) dt d\mx d\lambda\nonumber\\
&- \int_{\R^+\times \R^d \times \R}\phi(t,\mx) \eta'''(\lambda) h_\eps \left( \eps \sum_{j=1}^{d}|\pa_{{x_j}}u_\eps|^{2}-\delta \sum_{j=1}^{d}\pa_{x_{j}}u_{\eps}\pa_{t,x_{j}}u_{\eps} \right) dt d\mx d\lambda\nonumber\\
&=:-\int_{\R^+\times \R^d \times \R} \left( \Gamma_{1}^{\eps}+\Gamma_{2}^{\eps}\right) \cdot \nabla \phi(t,\mx) \eta''(\lambda)dt d\mx d\lambda\nonumber\\
& \hspace*{8em} -\int_{\R^+\times \R^d \times \R} \left(\Gamma_{3}^{\eps}+\Gamma_{4}^{\eps}\right) \phi(t,\mx) \eta'''(\lambda) dt d\mx d\lambda.\nonumber
\end{align} 
Now, since $\eta'\in C^2_c(\R)$ is arbitrary, if we put 
\begin{equation}
\label{G}
G_1^\eps =\Gamma_{1}^{\eps}+\Gamma_{2}^{\eps} \ \ {\rm and} \ \ G_2^\eps=\Gamma_{3}^{\eps}+\Gamma_{4}^{\eps}
\end{equation} 
we arrive at \eqref{kinetic}. 

%%%%%%%%%% (i) and (ii) %%%%%%%

%\todo{This I leave to the young forces :-) . Keep in mind conditions [C2] from the introduction. We take $n=\frac{1}{\eps^\alpha}$ in Proposition \ref{f-fe} i.e. we need a special approximation of $\mff$.}

Next, we prove (i) and (ii).  To begin with, we show that $\ddiv \Gamma_{1}^{\eps}\rightarrow 0\enspace\mbox{as} \enspace \eps\rightarrow 0^{+}$ in $H^{-1,r}_{loc}(\R^+\times \R^d\times \R)$, $r\leq 2$.

We have for any $\phi \in C_c^{1}(\R^+\times \R^d \times \R)$ supported in the hypercube with side length $M$ centered at zero:

%Namely, in a similar manner as we did before, we get \melanietodo{I changed the next eqns as the $\Gamma_i^\eps $ do not themselves contain $\eta $! Also i added a div since the original estimates were actually for $\ddiv G_1^\eps $ } 
\begin{equation}
\begin{split}
|\left<\ddiv \Gamma_{1}^{\eps},\phi \right>|&= \left|\sum_{j=1}^{d}\int_{0}^{T}\int_{\R^{d+1}}\eps h_\eps u_{x_j}^{\eps}\phi_{x_j}(t,\mx,
\lambda)d\mx d\lambda dt\right| \\
&\leq \eps M^{1/2} \| \phi\|_{H^1(\R^+\times \R^d \times \R)} \, \| \nabla u^\eps \|_{L^2([0,T]\times \R^d)}.
\end{split}
\end{equation}
%
%\\
%&\leq \sqrt{\eps}C\|\phi\|_{H_{0}^{1}([0,T]\times \R^{d})}\rightarrow 0^{+},\enspace as\enspace \eps\rightarrow 0^{+}
%Since the function $\eta$ is strictly convex, we conclude that $\Gamma_{3}^{\eps}\leq 0$.

To proceed further, we use \eqref{druga}. For estimating the term $\|\ddiv (\mff_\varepsilon)\|_{L^1(\R^{d+1})}$ appearing in \eqref{druga} 
we use approximations of the form $\mff_\eps :=\mff \star \omega_{n(\eps )}$  as in \eqref{mollifying} 
with $n(\eps)$ satisfying \eqref{neps}. 
Since $\delta=o(\eps^2)$, the $\mff_\eps$ will converge to $\mff $ as $\eps \to 0$ in $L^r_{loc}(\R^d\times \R)$ 
for any $1\le r \le p$. Using (C2) and (C3), we get

\begin{equation}\label{divfepsL1}
 \begin{split}
\|\ddiv (\mff_\varepsilon)\|_{L^1(\R^{d+1})}&= \| |\ddiv (\mff)|\star \omega_{n(\eps)}\|_{L^1(\R^{d+1})} \\
&\leq 
\left\| \frac{\mu(\mx)}{1+|\lambda|^{1+\beta}}\star \omega_{n(\eps)}(\mx,\lambda) \right\|_{L^1(\R^{d+1})} 
\\&\leq C \mu(\R^{d}) \| \omega_{n(\eps)} \|_{L^1(\R^{d+1})} \leq C,
\end{split} 
\end{equation}
where here and below, $C$ denotes a generic constant that may alter from line to line. Thus (keeping in mind \eqref{druga} 
and \eqref{cond}), from \eqref{divfepsL1} we conclude 
\begin{align}
|\left<\ddiv \Gamma_{1}^{\eps},\phi  \right>| \leq \sqrt{\eps} (n(\eps)^{-1}\sqrt{\delta}+1)C   
\|\phi\|_{H^{1}(\R^+\times \R^d \times \R})\rightarrow 0^{+},\enspace as\enspace \eps\rightarrow 0^{+}
\end{align} 
by \eqref{neps}.

For the next estimate we have
\begin{equation}\label{divgammaeps2}
\begin{split}
|\left<\ddiv \Gamma_{2}^{\eps},\phi \right>|\leq 2 M^{1/2}
\|\phi\|_{H_{0}^{1}([0,T]\times\R^{d+1})} \delta \|\nabla u^{\eps}_{t}\|_{L^{2}([0,T],L^{2}(\R^{d}))}.
\end{split}
\end{equation}

Comparing with \eqref{peta}, we see that it remains to estimate the terms  $\|\pa_\lambda \mff_\eps\|_{\infty} $ 
and $\|\ddiv (\mff_\varepsilon)\|_{L^2(\R^{d},L^\infty (\R))}$. For the first one we have
\begin{equation}\label{palambdafepsinfty}
\begin{split}
\|\pa_\lambda &\mff_\eps \|_{\infty} = \sup_{i} \| (\pa_\lambda f_i)\star \omega_{n(\eps)}\|_{L^\infty(\R^{d+1})}\\ 
&= \sup_{i} \| (\pa_\lambda f_i)\star \omega_{n(\eps)}\|_{L^\infty(\R^{d}, L^\infty(\R))} \\ &\leq \sup_{i} \| \pa_\lambda f_i\|_{L^1(\R^{d}, L^\infty(\R))}  \| \omega_{n(\eps)} \|_{L^\infty(\R^{d+1})}  \leq C n(\eps)^{-d-1} 
\end{split}
\end{equation} 
by (C3). For the second one we obtain from (C2) and the Young inequality
\begin{equation}\label{divfepsL2}
\begin{split}
\|\ddiv (\mff_\varepsilon)\|_{L^2(\R^{d},L^\infty (\R))}&\leq 
\| \frac{\mu(\mx)}{1+|\lambda|^{1+\beta}}\star \omega_{n(\eps)}(\mx,\lambda) \|_{L^2(\R^{d},L^\infty(\R))}\\
& =  \| \mu\star \omega_{n(\eps)}^{(1)}(\mx)\frac{1}{1+|\lambda|^{1+\beta}}\star \omega^{(2)}_{n(\eps)}(\lambda) \|_{L^2(\R^{d},L^\infty(\R))} \\
&\leq C n(\eps)^{-1}\|\mu \star \omega^{(1)}_{n(\eps)} \|_{L^2(\R^d)}  
\\ &\leq C n(\eps)^{-1}\mu(\R^d) \| \omega^{(1)}_{n(\eps)} \|_{L^2(\R^d)} \leq C 
n(\eps)^{-d/2-1}.
\end{split}
\end{equation}
Combining \eqref{divgammaeps2}, \eqref{palambdafepsinfty}, \eqref{peta}, \eqref{divfepsL1} and \eqref{cond} we see that 
\begin{align}
\nonumber
|\langle\ddiv & \Gamma_{2}^{\eps},\phi \rangle| 
\leq \\ 
& C \delta^{1/2} \left(\frac{1}{\sqrt{\eps} n(\eps)^{d+1}}
+ \Big(\frac{\sqrt{\delta}}{\sqrt{\eps}n(\eps)^{d+1}}+\sqrt{\eps}\Big)
\frac{1}{n(\eps)} + \frac{1}{\sqrt{\eps}n(\eps)^{d+1}}
+\frac{1}{n(\eps)^{d/2 + 1}} \right) \nonumber \\
&\cdot\|\phi\|_{H_{0}^{1}([0,T]\times \R^{d+1})}
\rightarrow 0^{+},\enspace as\enspace \eps\rightarrow 0^{+}.
\label{G2-new}
\end{align}  
by \eqref{neps}.
This implies claim (i) of the current lemma.
%\\
%&\leq C(d)\sqrt{\frac{\delta}{\eps}}\|\phi\|_{H_{0}^{1}([0,T]\times\R^{d})}\rightarrow 0^{+}, \eps\rightarrow 0^{+}.
Next, using \eqref{druga}, \eqref{peta}, \eqref{cond}, \eqref{divfepsL1} and \eqref{palambdafepsinfty} we estimate 
\begin{equation}
\label{G4-new}
\begin{split}
|\left<\Gamma_{4}^{\eps},\phi \right>|& \leq \delta\left|\sum_{j=1}^{d}\int_{0}^{T}\int_{[-M,M]^{d+1}}|u_{x_j}^{\eps}| \, |u_{t,x_{j}}^{\eps}| \, |\phi(t,\mx,\lambda)|d\mx dt d\lambda \right|\\
&\leq M \delta \|\phi\|_{\infty}\sum_{j=1}^{d}\int_{0}^{T}\int_{\R^{d}}|u_{x_j}^{\eps}||u_{t,x_{j}}^{\eps}|dx dt\\
&\leq M \delta \|\phi\|_{\infty} \|\nabla u^{\eps}\|_{L^{2}([0,T]\times \R^{d})} \|\nabla u_{t}^{\eps}\|_{L^{2}([0,T]\times \R^{d})}\\
& \le C \sqrt{\delta} \|\phi\|_\infty  \Big(\frac{1}{\sqrt{\eps}} +
\frac{{\sqrt{\delta}}}{\sqrt{\eps} n(\eps)}\Big) \Big[\frac{1}{\sqrt{\eps}n(\eps)^{d+1}}+ \\
& + \Big(\frac{\sqrt{\delta}}{\sqrt{\eps}n(\eps)^{d+1}}
+\sqrt{\eps}\Big)\frac{1}{n(\eps)} + 
\frac{1}{\sqrt{\eps}n(\eps)^{d+1}} +  \frac{1}{n(\eps)^{d/2+1}}
\Big],  
\end{split}
\end{equation}
which goes to $0$ as $\eps\to 0+$ by \eqref{neps}.

Finally, it remains to estimate $\Gamma_3^\eps$. Using \eqref{druga}, \eqref{cond} and \eqref{divfepsL1} we obtain 
\begin{equation}
\begin{split}
|\left<\Gamma_{3}^{\eps},\phi \right>|& \leq \eps  \left|\sum_{j=1}^{d}\int_{0}^{T}\int_{[-M,M]^{d+1}} |u_{x_j}^{\eps}(t,\mx)|^2 \phi(t,\mx,\lambda) \, d\mx dt d\lambda\right|\\
&\leq  \eps M \|\phi\|_{\infty} \|\nabla u^{\eps}\|^2_{L^{2}([0,T]\times \R^{d})} \\
&\leq C \|\phi\|_{\infty} (1+\frac{\delta}{n(\eps)^2}) \leq  C \|\phi\|_{\infty}.
\end{split}
\end{equation} 

From the estimates given above, we see that $\ddiv\Gamma_1^\eps$ and $\ddiv\Gamma_2^\eps$ converge strongly to zero in $H^{-1}_{loc}(\R^+\times \R^d \times \R)$ implying that they converge strongly 
in $W^{-1,r}_{loc}(\R^+\times \R^d \times \R)$, $r\in [1,2)$, as well.

As for  $\Gamma_3^\eps$ and $\Gamma_4^\eps$, they are locally bounded in the space of Radon measures 
(which we denote it by ${\cal M}(\R^+\times \R^d \times \R)$) and thus, they are precompact in 
$W^{-1,r}_{loc}(\R^+\times \R^d \times \R)$, $r\in \big[1,\frac{d+2}{d+1}\big)$ 
\cite[Th.\ 1.3.2]{evans}. 
This implies condition (ii) from the Lemma.

%\melanietodo{I finished the estimates, but I don't know how exactly they imply points (i) and (ii).}

\end{proof}

According to the previous theorem, we see that we can apply the velocity averaging lemma Theorem \ref{va}. Indeed, we have:

\begin{lemma}
Any solution $h_\eps$ of \eqref{kinetic} satisfies 
\begin{equation}
\label{for-va}
\pa_t h_\eps +\ddiv (h_\eps \pa_\lambda \mff (\mx,\lambda) )=\pa_\lambda \ddiv (G_1^\eps)+ \pa^2_\lambda G^\eps_2 + \ddiv(G^\eps_3)+\pa_\lambda  (h_\eps \ddiv \mff_\eps(\mx,\lambda)),
\end{equation} where $(G_j^\eps)$, $j=1,3$, are strongly precompact in $L_{loc}^r(\R^+\times \R^d \times \R)$ while $(G^\eps_2)$ and $(h_\eps \ddiv \mff_\eps(\mx,\lambda))$ are strongly precompact in $W^{-1,r}_{loc}(\R^+\times \R^d \times \R)$, $r\in \big[1,\frac{d+2}{d+1}\big)$.
\end{lemma}
\begin{proof}
It is enough to rewrite \eqref{kinetic} in the form
\begin{equation*}
\begin{split}
& \frac{1}{2}\pa_t h_\eps +\ddiv (h_\eps \pa_\lambda \mff (\mx,\lambda) )
=\pa_\lambda \ddiv (G_1^\eps)+ \pa^2_\lambda G^\eps_2  \\& 
\hspace*{5em}
+\ddiv (h_\eps (\pa_\lambda \mff (\mx,\lambda)-\pa_\lambda \mff_\eps (\mx,\lambda)))+\pa_\lambda  (h_\eps \ddiv \mff_\eps(\mx,\lambda)),
\end{split}
\end{equation*} and denote
$$
G_3^\eps=h_\eps (\pa_\lambda \mff (\mx,\lambda)-\pa_\lambda \mff_\eps (\mx,\lambda)).
$$ Clearly, since $\mff_\eps=K_\eps\cdot\mff \star \omega_{n(\eps)}$, we have $G_3^\eps \to 0$ as $\eps\to 0$ in $L_{loc}^p(\R^d\times \R)$ 
(were $p$ is given in (C1)). 
On the other hand, according to condition (C2), we see that $(h_\eps \ddiv \mff_\eps(\mx,\lambda)))$ is bounded in the space of measures and thus strongly precompact in $ W^{-1,r}_{loc}(\R^+\times \R^d \times \R)$.
This concludes the proof.
\end{proof} Now, we are ready to use Theorem \ref{va}.  
In fact, we may choose $r\in \big(1,\frac{d+2}{d-1}\big)$ and $s\ge 2$
in such a way that 
\[
1+\frac{1}{s} - \frac{1}{r} <1
\]
and then set $\bar p := (1+\frac{1}{s} - \frac{1}{r})^{-1}$. Then fixing $\bar{p}'$ such that $\frac{1}{\bar p} + \frac{1}{\bar p'}=1$,
the assumptions of Theorem \ref{main-theorem} precisely allow to apply Theorem \ref{va}
(noting that $(h_\eps)$ is bounded in any $L^s_{loc}$, hence contains a weakly convergent
subsequence).
More precisely, one easily checks that any $\bar p\in \big(\frac{2d+4}{d+4},\infty\big)$, and 
thereby any $\bar p' \in \big(1,2+\frac{4}{d}\big)$ can be obtained in this way.
This is the reason for the specific assumption in Theorem \ref{main-theorem}.

Consequently, for any sequence $\eps_n\searrow 0$ and any
$\rho \in L^2(\R)$, setting $h_n:=h_{\eps_n}$ and $u_n:=u_{\eps_n}$, 
the sequence $\left(\int_\R \rho(\lambda) h_n(t,\mx,\lambda) d\lambda \right) $  
is strongly precompact in $L^1_{loc}(\R^+\times \R^d)$. As we shall see in the next theorem, this implies strong 
convergence of the sequence $(u_n)$. We first need the following auxiliary result.

\begin{lemma}
\label{lconv-1}
Assume that the sequence $(u_n)$ is bounded in $L^p(\Omega)$, $\Omega \Subset \R^d$, $d\in \N$, for some $p\geq 1$. 
Define
$$
\Omega_n^l=\{ \mx\in\Omega:\, |u_{n}(\mx)| > l \}.
$$ Then
\begin{equation}
\label{conv1}
\lim\limits_{l\to \infty} \sup\limits_{n\in \N} {\rm
meas}(\Omega_n^l) =0.
\end{equation}

\end{lemma}
\begin{proof}
Since $(u_{n})$ is bounded in $L^p(\Omega)$ and thus in $L^1(\Omega)$ as well (since $\Omega \Subset \R^m$), we have
\begin{align*}
&\sup\limits_{n\in \N} \int_{\Omega }|u_{n}(\mx)| d\mx \geq
\sup\limits_{n\in \N} \int_{\Omega_n^l} l d\mx\, \implies \frac{1}{l}\sup\limits_{k\in \N} \int_{\Omega }|u_{n}(\mx)| d\mx \geq
\sup\limits_{n\in \N} {\rm meas}(\Omega_n^l),
\end{align*} implying \eqref{conv1} after letting $l\to \infty$ here.
\end{proof}
With the above notations we finally arrive at:
\begin{theorem}
Under the non-degeneracy conditions \eqref{non-deg}, the sequence $(u_n=u_{\eps_n})$ of solutions to \eqref{vc-1}, \eqref{id-pert} strongly converges along a subsequence in $L^1_{loc}(\R^+\times \R^d)$.
\end{theorem}
\begin{proof} 
For $l\in \N$ we set $\rho(\lambda)=\chi_{(-l,l)}(\lambda)$, where $\chi_{(-l,l)}$ is the characteristic function of the interval $(-l,l)$. 
Then by Theorem \ref{va} combined with a diagonalization argument there exists a common subsequence, 
again denoted by $(u_n)$, 
such that for any $l\in \N$:
\begin{align}
\frac{1}{2}\int_{-l}^l h_n(t,\mx,\lambda) d\lambda&=u_n
\chi_{\{|u_n|\leq l\}}(t,\mx)+l \chi_{\{u_n>l\}}(t,\mx)-l \chi_{\{u_n<-l\}}(t,\mx)\nonumber\\
&=:T_l(u_n) \to u^l \label{3.8first}
\end{align}
as $n\to \infty$ in $L^1_{loc}(\R^+\times \R)$. 
The operators $T_l$ are known as truncation operators \cite{DHM}. 

It is not difficult to prove that from here we can conclude about the convergence of $(u_n)$. 
First, we  show that the sequence $(u^l)$  converges strongly in
$L^1_{loc}(\R^+\times \R)$ as $l\to \infty$.

To this end, let $\Omega\Subset \R^+\times \R^d$. We claim that
\begin{equation}
\label{uni_n}
\lim_l \sup_n \|{T_l(u_{n}) -  u_{n}}\|_{L^1(\Omega)} \to 0\,.
\end{equation}
In fact, let
$$
\Omega_n^l=\{ (t,\mx)\in\Omega:\, |u_{n}(t,\mx)| > l \}.
$$
Then since (by \eqref{prva}, \eqref{cond}, \eqref{neps}, and \eqref{divfepsL1})
 $(u_{n})$ is bounded in $L^2(\Omega)$, we have

$$
\int_\Omega |u_{n}-T_l(u_{n})| dt d\mx
\leq \int_{\Omega_n^l} |u_{n}|dt d\mx  \leq {\rm meas}(\Omega_n^l)^{1/2} \, \|u_n\|_{L^2(\Omega)}\to 0
$$ 
as $l\to \infty$, uniformly with respect to $n$ according to \eqref{conv1}.
This gives \eqref{uni_n}.

Next, we estimate
\begin{align}
\label{ul} \|u^{l_1}-u^{l_2}\|_{L^1( \Omega)}
&\leq
\|u^{l_1}-T_{l_1}(u_n)\|_{L^1( \Omega)}+\|T_{l_1}(u_n)-u_n\|_{L^1( \Omega)}\nonumber\\
&+\|T_{l_2}(u_n)-u_n\|_{L^1( \Omega)}+\|T_{l_2}(u_n)-u^{l_2}\|_{L^1( \Omega)}\,,\nonumber
\end{align}
which together with \eqref{3.8first} and \eqref{uni_n} implies that $(u^l)$ is a Cauchy sequence. Thus, there exists $u\in L^1(\Omega)$ such that
\begin{equation}
\label{conv2}
u^l \to u \ \ {\rm in} \ \ L^1(\Omega).
\end{equation}

Now it is not difficult to see that the entire sequence $(u_n)$ converges towards $u$
in $L^1(\Omega)$ as well. Namely, 
\begin{align*}
 % \label{dbk11}
  \|u_n-u\|_{L^1(\Omega)} \leq
\|u_n-T_l(u_n)\|_{L^1(\Omega)}+\|T_l(u_n)-u^l\|_{L^1(\Omega)}+\|u^l-u\|_{L^1(\Omega)},
\end{align*}
which by the definition of the functions $u^l$, in conjunction with 
\eqref{uni_n} and \eqref{conv2} gives the claim. \end{proof}

We actually proved only item (i) of Theorem \ref{main-theorem}. The other two items (item (ii) and item (iii)) can be proven by an adaptation of the proof for item (i). We provide a more precise explanation in the following remark.

\begin{remark}
\label{39}
Derivation of (ii) and (iii) from Theorem \ref{main-theorem}.

If we additionally assume that $\|\pa_\lambda \mff\| \leq C <\infty$ then we can use \eqref{peta} in the case $\delta={\cal O}(\eps^2)$ and $\frac{\sqrt{\delta}}{\sqrt{\eps} n(\eps)^{d/2+2}}\to 0$ and the considerations thereafter will remain the same. This gives (ii). Indeed, we have to check whether the sequence of equations \eqref{for-va} satisfies the conditions of Theorem \ref{va}. To this end, we need to estimate $\Gamma^\eps_j$, $j=1,\dots,4$, appearing in \eqref{G}. If we assume $\|\pa_\lambda \mff\| \leq C <\infty$ then we can omit \eqref{palambdafepsinfty} and get instead of \eqref{G2-new}
\begin{align}
\label{G2-newReg}
|\left<\ddiv \Gamma_{2}^{\eps},\phi \right>| 
\leq C \delta^{1/2} \left(\frac{1}{\sqrt{\eps} }
+ \Big(\frac{\sqrt\delta}{\sqrt{\eps}}+\sqrt{\eps}\Big)\frac{1}{n(\eps)} + \frac{1}{\sqrt{\eps}}
+\frac{1}{n(\eps)^{d/2 +1}} \right) \\
\cdot\|\phi\|_{H_{0}^{1}([0,T]\times \R^{d+1})}
\rightarrow 0^{+},\enspace as\enspace \eps\rightarrow 0^{+}.
\nonumber
\end{align} Similarly, instead of \eqref{G4-new}, we have
\begin{equation}
\label{G4-newReg}
\begin{split}
|\left<\Gamma_{4}^{\eps},\phi \right>|& \leq \delta\left|\sum_{j=1}^{d}\int_{0}^{T}
\int_{[-M,M]^{d+1}}|u_{x_j}^{\eps}| \, |u_{t,x_{j}}^{\eps}| \, |\phi(t,\mx,\lambda)|d
\mx dt d\lambda \right|\\
& \le C \sqrt{\delta} \|\phi\|_\infty \Big(\frac{1}
{\sqrt{\eps}} +
\frac{\sqrt{\delta}}{\sqrt{\eps} n(\eps)}\Big) \Big[\frac{1}{\sqrt{\eps}} 
+ \\
& + \Big(\frac{\sqrt{\delta}}{\sqrt{\eps}}
+\sqrt{\eps}\Big)\frac{1}{n(\eps)} + 
\frac{1}{\sqrt{\eps}}+  \frac{1}{n(\eps)^{d/2 +1}}
\Big]\leq C \|\phi\|_\infty 
\end{split}
\end{equation}

If we have a regular flux i.e. $\mff \in C^1(\R^d\times \R)$ 
and the conditions $\delta=o(\eps^2)$ and 
$\frac{\sqrt{\delta}}{\sqrt{\eps} n(\eps)^{d/2+2}}\to 0$ as $\eps\to 0$, we can use Young measures in the way given in \cite{Lfl-b} to derive the convergence. Actually, in this case the diffusion given by the second order term in \eqref{vc-1} will dominate over the dynamic capillarity given by the third order term in \eqref{vc-1} and therefore, we will end up with the unique Kruzhkov admissible solution to \eqref{cl-1} with the corresponding initial data. 

Indeed, it is not difficult to see that under the regularity assumptions on $\mff$ and $\delta=o(\eps^2)$, we have
\begin{equation}
\label{regular}
\begin{split}
& -\int_{\R^+\times \R^d \times \R} \left( \Gamma_{1}^{\eps}+\Gamma_{2}^{\eps}\right) \cdot \nabla \phi(t,\mx) \eta''(\lambda)dt d\mx d\lambda \to 0\\
&  -\int_{\R^+\times \R^d \times \R} \Gamma_{4}^{\eps}(t,\mx,\lambda) \phi(t,\mx) \eta'''(\lambda) dt d\mx d\lambda \to 0.
\end{split}
\end{equation} Now, denote by $\nu_{(t,\mx)}$ the Young measure (e.g. \cite{evans}) corresponding to a subsequence of the family $(u^\eps)$. If we let $\eps\to 0$ in \eqref{kinetic-new} along the subsequence for a convex entropy $\eta$, and take \eqref{regular} into account, we get in the sense of distributions (below, $q(\mx,\lambda)=\int_{-\infty}^\lambda \eta'(v) \pa_v \mff(\mx,v) dv$ is the entropy flux)

\begin{equation*}
\partial_{t} \int \eta(\lambda)d\nu_{(t,\mx)}(\lambda) +\ddiv \int q(\mx,\lambda) d\nu_{(t,\mx)}(\lambda)+\int \int_{-\infty}^{\lambda} \eta''(v)\ddiv \mff(\mx,v) dv d\nu_{(t,\mx)}(\lambda) \leq 0.
\end{equation*} Now, we simply rely on the result from \cite{Dpe} (see also \cite{KLfl2, Sze}) to conclude that the Young measure is unique and atomic i.e. of the form $\nu_{(t,\mx)}(\lambda)=\delta(\lambda-u(t,\mx))$ for the unique entropy admissible solution to the underlying conservation law \eqref{cl-1}.

\end{remark}

{\bf Acknowledgment} The work is supported in part by project P30233 of the Austrian Science Fund (FWF) and by the Lise Meitner project M 2669 of the Austrian Science Fund (FWF). It is also supported in part by the Croatian Science Foundation under Project MiTPDE (number IP-2018-01-2449).

\end{document}